\documentclass[12pt]{amsart}
\usepackage[margin=1in]{geometry}
\usepackage{amsmath,amsfonts,amsthm,amssymb,bbm}
\usepackage{graphicx,color,dsfont}
\usepackage{enumitem}
\usepackage{fourier}

\newtheorem{theorem}{Theorem}
\newtheorem{conjecture}{Conjecture}
\newtheorem{lemma}{Lemma}
\newtheorem{proposition}{Proposition}

\newtheorem{definition}{Definition}
\newtheorem{corollary}{Corollary}

\newtheorem{claim}{Claim}

 \theoremstyle{definition}
 
 \theoremstyle{remark}

 \numberwithin{equation}{section}

\newcommand{\vertiii}[1]{{\left\vert\kern-0.25ex\left\vert\kern-0.25ex\left\vert #1
    \right\vert\kern-0.25ex\right\vert\kern-0.25ex\right\vert}}

\newcommand{\f}[2]{\frac{#1}{#2}}

\newcommand{\cl}{{\mathcal L}}

\begin{document}

\newcommand{\abs}[1]{\lvert#1\rvert}

\newcommand{\al}{\alpha}
\newcommand{\be}{\beta}
\newcommand{\wh}[1]{\widehat{#1}}
\newcommand{\ga}{\gamma}
\newcommand{\Ga}{\Gamma}
\newcommand{\de}{\delta}
\newcommand{\ben}{\beta_n}
\newcommand{\De}{\Delta}
\newcommand{\ve}{\varepsilon}
\newcommand{\ze}{\zeta}
\newcommand{\Th}{\Theta}
\newcommand{\ka}{\kappa}
\newcommand{\la}{\lambda}
\newcommand{\laj}{\lambda_j}
\newcommand{\lak}{\lambda_k}
\newcommand{\La}{\Lambda}
\newcommand{\si}{\sigma}
\newcommand{\Si}{\Sigma}
\newcommand{\vp}{\varphi}
\newcommand{\om}{\omega}
\newcommand{\Om}{\Omega}

\newcommand{\ro}{{\mathbf R}}
\newcommand{\rn}{{\mathbb R}^n}
\newcommand{\rd}{{\mathbf R}^d}
\newcommand{\rmm}{{\mathbf R}^m}
\newcommand{\rone}{\mathbb R}
\newcommand{\rtwo}{\mathbf R^2}
\newcommand{\rthree}{\mathbf R^3}
\newcommand{\rfour}{\mathbf R^4}
\newcommand{\ronen}{{\mathbf R}^{n+1}}
\newcommand{\ku}{\mathbf u}
\newcommand{\kw}{\mathbf w}
\newcommand{\kf}{\mathbf f}
\newcommand{\kz}{\mathbf z}

\newcommand{\N}{\mathbf N}

\newcommand{\tn}{\mathbf T^n}
\newcommand{\tone}{\mathbf T^1}
\newcommand{\ttwo}{\mathbf T^2}
\newcommand{\tthree}{\mathbf T^3}
\newcommand{\tfour}{\mathbf T^4}

\newcommand{\zn}{\mathbf Z^n}
\newcommand{\zp}{\mathbf Z^+}
\newcommand{\zone}{\mathbf Z^1}
\newcommand{\zz}{\mathbf Z}
\newcommand{\ztwo}{\mathbf Z^2}
\newcommand{\zthree}{\mathbf Z^3}
\newcommand{\zfour}{\mathbf Z^4}

\newcommand{\hn}{\mathbf H^n}
\newcommand{\hone}{\mathbf H^1}
\newcommand{\htwo}{\mathbf H^2}
\newcommand{\hthree}{\mathbf H^3}
\newcommand{\hfour}{\mathbf H^4}

\newcommand{\cone}{\mathbf C^1}
\newcommand{\ctwo}{\mathbf C^2}
\newcommand{\cthree}{\mathbf C^3}
\newcommand{\cfour}{\mathbf C^4}
\newcommand{\dpr}[2]{\langle #1,#2 \rangle}

\newcommand{\sn}{\mathbf S^{n-1}}
\newcommand{\sone}{\mathbf S^1}
\newcommand{\stwo}{\mathbf S^2}
\newcommand{\sthree}{\mathbf S^3}
\newcommand{\sfour}{\mathbf S^4}

\newcommand{\lp}{L^{p}}
\newcommand{\lppr}{L^{p'}}
\newcommand{\lqq}{L^{q}}
\newcommand{\lr}{L^{r}}
\newcommand{\echi}{(1-\chi(x/M))}
\newcommand{\chip}{\chi'(x/M)}

\newcommand{\wlp}{L^{p,\infty}}
\newcommand{\wlq}{L^{q,\infty}}
\newcommand{\wlr}{L^{r,\infty}}
\newcommand{\wlo}{L^{1,\infty}}
\newcommand{\eps}{\epsilon}

\newcommand{\lprn}{L^{p}(\rn)}
\newcommand{\lptn}{L^{p}(\tn)}
\newcommand{\lpzn}{L^{p}(\zn)}
\newcommand{\lpcn}{L^{p}(\cn)}
\newcommand{\lphn}{L^{p}(\cn)}

\newcommand{\lprone}{L^{p}(\rone)}
\newcommand{\lptone}{L^{p}(\tone)}
\newcommand{\lpzone}{L^{p}(\zone)}
\newcommand{\lpcone}{L^{p}(\cone)}
\newcommand{\lphone}{L^{p}(\hone)}

\newcommand{\lqrn}{L^{q}(\rn)}
\newcommand{\lqtn}{L^{q}(\tn)}
\newcommand{\lqzn}{L^{q}(\zn)}
\newcommand{\lqcn}{L^{q}(\cn)}
\newcommand{\lqhn}{L^{q}(\hn)}

\newcommand{\lo}{L^{1}}
\newcommand{\lt}{L^{2}}
\newcommand{\li}{L^{\infty}}

\newcommand{\co}{C^{1}}
\newcommand{\ci}{C^{\infty}}
\newcommand{\coi}{C_0^{\infty}}

\newcommand{\ca}{\mathcal A}
\newcommand{\cs}{\mathcal S}
\newcommand{\cg}{\mathcal G}
\newcommand{\cm}{\mathcal M}
\newcommand{\cf}{\mathcal F}
\newcommand{\cb}{\mathcal B}
\newcommand{\ce}{\mathcal E}
\newcommand{\cd}{\mathcal D}
\newcommand{\cn}{\mathcal N}
\newcommand{\cz}{\mathcal Z}
\newcommand{\crr}{\mathcal R}
\newcommand{\cc}{\mathcal C}
\newcommand{\ch}{\mathcal H}
\newcommand{\cq}{\mathcal Q}
\newcommand{\cp}{\mathcal P}
\newcommand{\cx}{\mathcal X}

\newcommand{\pv}{\textup{p.v.}\,}
\newcommand{\loc}{\textup{loc}}
\newcommand{\intl}{\int\limits}
\newcommand{\iintl}{\iint\limits}
\newcommand{\dint}{\displaystyle\int}
\newcommand{\diint}{\displaystyle\iint}
\newcommand{\dintl}{\displaystyle\intl}
\newcommand{\diintl}{\displaystyle\iintl}
\newcommand{\liml}{\lim\limits}
\newcommand{\suml}{\sum\limits}
\newcommand{\ltwo}{L^{2}}
\newcommand{\supl}{\sup\limits}
\newcommand{\df}{\displaystyle\frac}
\newcommand{\p}{\partial}
\newcommand{\Ar}{\textup{Arg}}
\newcommand{\abssigk}{\widehat{|\si_k|}}
\newcommand{\ed}{(1-\p_x^2)^{-1}}
\newcommand{\tT}{\tilde{T}}
\newcommand{\tV}{\tilde{V}}
\newcommand{\wt}{\widetilde}
\newcommand{\Qvi}{Q_{\nu,i}}
\newcommand{\sjv}{a_{j,\nu}}
\newcommand{\sj}{a_j}
\newcommand{\pvs}{P_\nu^s}
\newcommand{\pva}{P_1^s}
\newcommand{\cjk}{c_{j,k}^{m,s}}
\newcommand{\Bjsnu}{B_{j-s,\nu}}
\newcommand{\Bjs}{B_{j-s}}
\newcommand{\Ly}{L_i^y}
\newcommand{\dd}[1]{\f{\partial}{\partial #1}}
\newcommand{\czz}{Calder\'on-Zygmund}
\newcommand{\chh}{\mathcal H}

\newcommand{\lbl}{\label}
\newcommand{\beq}{\begin{equation}}
\newcommand{\eeq}{\end{equation}}
\newcommand{\beqna}{\begin{eqnarray*}}
\newcommand{\eeqna}{\end{eqnarray*}}
\newcommand{\beqn}{\begin{equation*}}
\newcommand{\eeqn}{\end{equation*}}
\newcommand{\bp}{\begin{proof}}
\newcommand{\ep}{\end{proof}}
\newcommand{\bprop}{\begin{proposition}}
\newcommand{\eprop}{\end{proposition}}
\newcommand{\bt}{\begin{theorem}}
\newcommand{\et}{\end{theorem}}
\newcommand{\bex}{\begin{Example}}
\newcommand{\eex}{\end{Example}}
\newcommand{\bc}{\begin{corollary}}
\newcommand{\ec}{\end{corollary}}
\newcommand{\bcl}{\begin{claim}}
\newcommand{\ecl}{\end{claim}}
\newcommand{\bl}{\begin{lemma}}
\newcommand{\el}{\end{lemma}}
\newcommand{\dea}{(-\De)^\be}
\newcommand{\naa}{|\nabla|^\be}
\newcommand{\cj}{{\mathcal J}}

\title[Solitons for the   NLS under  trapping  potential ]
 {Ground states for the  nonlinear Schr\"odinger equation under a general trapping  potential}

\thanks{ 
 Stanislavova is partially supported by  NSF-DMS under grant \# 1516245.   Stefanov    is partially  supported by  NSF-DMS under grant \# 1908626.}

\author[Milena  Stanislavova]{\sc Milena Stanislavova}
\address{ Department of Mathematics,
University of Kansas,
1460 Jayhawk Boulevard,  Lawrence KS 66045--7523, USA}
\email{stanis@ku.edu}

\author[Atanas Stefanov]{\sc Atanas G. Stefanov}
\address{ Department of Mathematics,
University of Kansas,
1460 Jayhawk Boulevard,  Lawrence KS 66045--7523, USA}
\email{stefanov@ku.edu}



\date{\today}

\begin{abstract}
The classical Schr\"odinger equation with a harmonic trap potential $V(x)=|x|^2$, describing the quantum harmonic oscillator, has been studied quite extensively in the last twenty  years. Its ground states are  bell-shaped and unique, among localized positive solutions. In addition, they have been shown  to be non-degenerate and (strongly) orbitally  stable. 
All of these results, produced over the course of many publications  and   multiple authors,  rely on  ODE methods specifically designed  for the Laplacian and the power function potential. 

In this article, we provide a wide generalization of these results. More specifically, we assume sub-Laplacian fractional dispersion and a very  general form of the trapping potential $V$, with  the driving linear operator in  the form $\ch=(-\De)^s+V, 0<s\leq 1$.  We show that the normalized waves of such semilinear fractional Schr\"odinger equation exist, they are bell-shaped, provided that the   non-linearity is of the form $|u|^{p-1} u, p<1+\f{4 s}{n}$. In addition, we show that such waves are non-degenerate, and strongly orbitally stable.  Most of these results are new even in the classical case $\ch=-\De+V$, where $V$ is a general trapping potential considered herein.

\end{abstract}
 
\maketitle

\section{Introduction} 
The  Schr\"odinger equation is an ubiquitous model in quantum mechanical applications.  In this work, we consider a model, in which the system is subjected to so-called magnetic traps, which keeps the action  very tightly to the trap. Mathematically, the probability density functions that arise as squares of the solutions have unusually high space localization, compared to the standard model without trapping. 
Next, we formally  introduce the model. 
\subsection{The model} 
We consider the fractional Schr\"odinger equation subject to a trapping harmonic potential
\begin{equation}
\label{a:10}
i u_t +(-\De)^s u+V(x) u-|u|^{p-1} u=0, (t,x)\in \rone_+\times \rn
\end{equation}
where $n\geq 1$, $p>1$ and we assume that the potential is trapping. That is 
\begin{definition}
	\label{de:trap} 
	We say that a potential $V:\rone_+\to \rone_+$ is trapping, if\footnote{ The requirement for at most polynomial growth of $V$ is likely just  a technicality, but we prefer to enforce it, due to the difficulties with the space of test functions, should $V$ has faster growth.}
	\begin{itemize}
		\item $V(x)=V(|x|)$, 
		\item $V$ is increasing and in fact, assume $V\in C^1(\rone_+)$, $V'(r)>0$. 
		\item $\lim_{r\to \infty} V(r)=+\infty$, but it has at most polynomial growth. That is, for some $N>1$, $V(r)\leq C(1+r)^N$. 
	\end{itemize} 
\end{definition}

 The natural energy space associated with this problem is  the space 
$$
X_s:=\dot{H}^{s}(\rn)\cap L^2(V(x) dx)=\{u:\rn\to \rone: \|(-\De)^{s/2} u\|_{L^2}^2+ \int_{\rn} V(x) |u(x)|^2  dx<\infty  \}
$$

  In typical quantum mechanical applications, $u$ is the probability density function of a particle trapped inside a trapping potential well, traditionally modeled by  $V(x)$. Note that the linear operator driving this particular evolution is 
$
\ch:=(-\De)^s+V. 
$

Quite a bit is known about $\ch$, we will just mention a few relevant properties. To that end, $\ch$ is a self-adjoint operator, when considered on the domain  
$$
D(\ch)=\dot{H}^{2s}(\rn)\cap L^2(V^2(x) dx)=\{u:\rn\to \rone: \|(-\De)^{s} u\|_{L^2}^2+ \int_{\rn} V^2(x) |u(x)|^2  dx<\infty  \}
$$
In addition, we will show in a rather standard manner, that its spectrum, which is of course all real, consists entirely of discrete eigenvalues of finite multiplicity, which converge to $+\infty$.  
Recall the conservation laws for \eqref{a:10}, the Hamiltonian energy 
$$
E[u]:=\f{1}{2}\left( \int_{\rn} |\nabla|^s u(t,x)|^2+   \int_{\rn} V(x) |u(t,x)|^2  dx\right) - \f{1}{p+1} \int_{\rn}  |u(t,x)|^{p+1} dx,
$$
and the $L^2$ norm ( or particle number or power) 
$$
P[u]=\int_{\rn}  |u(t,x)|^2 dx
$$
Standing waves of this equation are  solutions of \eqref{a:10} in  the form $u(t,x)=e^{- i \om t} \phi_\om(x)$.  Clearly, they   satisfy the elliptic equation 
\begin{equation}
\label{a:20}
(-\De)^s \phi +V  \phi +\om \phi -|\phi|^{p-1} \phi=0,  x\in  \rn
\end{equation}
for some $\om$.  We shall be particularly interested in positive solutions of \eqref{a:20}. In addition, we shall be interested in their dynamical stability properties.  

In the classical case of harmonic Schr\"odinger equation, that is $s=1$, $V(x)=|x|^2$, the problem is well-studied. This is of course the standard model\footnote{in non-dimensionalized variables} of the quantum harmonic oscillator. 
Most of the finding of this paper confirm these and present a natural extension to the more general case of potentials introduced in Definition \ref{de:trap} and the sub-Laplacian dispersion.  Thus, we take the opportunity to review the relevant recent results, which will also help us outline the areas of interest in this study. 
\subsection{The quantum harmonic oscillator} 
The linear quantum oscillator operator is given by $-\De+|x|^2$. It  has been studied in great detail over the last thirty years. In particular, it has been established that it is self-adjoint, with spectrum entirely consisting of eigenvalues of finite multiplicity. In fact, the eigenvalues  are explicitly known and even the corresponding eigenvectors can be written in terms of  the classical Hermite polynomials - for example, the lowest eigenvalue $\si_0(-\De+|x|^2)=n$, with corresponding eigenfunction $e^{-\f{|x|^2}{2}}$. 

Regarding the issues of interest in this work, for the  corresponding Schr\"odinger problem 
\begin{equation}
\label{10}
i u_t -\De u+|x|^2 u-|u|^{p-1} u=0, (t,x)\in \rone_+\times \rn,
\end{equation}
standing wave solutions, namely solutions, as above $u=e^{- i \om t} \phi$ can be constructed. More precisely, one is (initially)  looking for distributional   solutions,  that is $\phi \in X_1=H^1(\rn)\cap L^2(|x|^2 dx)$,   so that 
\begin{equation}
\label{20}
-\De \phi +|x|^2 \phi +\om \phi -|\phi|^{p-1} \phi=0,  x\in  \rn,
\end{equation}
in a distributional sense. 
For example, it is well-known that for any $\om \in (-n, \infty)$ and $1<p< p^*_n:=\left\{\begin{array}{cc} 
+\infty & n=1,2 \\
1+ \f{4}{n-2} & n\geq 3
\end{array}
\right.$
there exists  solutions of \eqref{20}, which belong to the energy space $X_1$, see \cite{Ca, HO, KT}.  Here  the significance of the restriction $\om>-n$ is in that $\ch+\om\geq (\om+n) Id>0$. 
In addition, very strong uniqueness theorems for \eqref{20} are known, if we restrict our attention to ground states - that is, positive solutions of \eqref{20}. 
 Let us state the uniqueness and non-degeneracy results, already  available in the literature. 
 \begin{proposition}
 	\label{prop:lk}
 	Let $n\geq 2$ and $1<p<p_n^*$. For every $\om>-n$, there is an unique positive solution $\phi_\om: \lim_{|x|\to \infty}\phi_\om(x)=0$,  of 
 	$$
 	-\De \phi +|x|^2 \phi +\om \phi - \phi^p=0,  x\in  \rn. 
 	$$
 	 Moreover, such solution is non-degenerate, that is the linearized operator $\cl_+:=-\De+|x|^2 + \om - p \phi^{p-1}$ has a trivial kernel, $Ker[\cl_+]=\{0\}$. 
 \end{proposition}
 For the proof of the uniqueness, we refer to \cite{HO, HO1, KT}. The non-degeneracy was established in  \cite{KT} and in a more general form, \cite{BO}. 
 We now review the known stability results for the ground states of \eqref{10}.  In the $L^2$ subcritical range,  $1<p<1+\f{4}{n}$, the ground states have been constructed  in \cite{Zhang}, together with the weak stability properties. This, together with the uniqueness yields the strong orbital stability for these waves\footnote{Although it looks as if this result  has not been stated explicitly in the literature}.  In addition,  the stability is known for the waves with  any $p\in (1, p_n^*)$,  $-n<\om<-n+\eps, 0<\eps<<1$, \cite{FuOh1}. On the other hand, there exists $N>>1$, so that for $\om>N$, the ground states $\phi_\om$ are   unstable for $1+\f{4}{n}<p<p_n^*$, \cite{Fu, FuOh1, FuOh2}.  
 
 We should mention that there are various  results on blow up for \eqref{10}, for generic initial data (not necessarily related to solitary waves),  for example  in the papers \cite{Ca, Zhang}.  Instability by blow up was unknown till the work of Ohta, \cite{Ohta}, who has shown that if $p>1+\f{4}{n}$, there exists $\om_{p,n}$, so that all solitons in the regime $\om>\om_{p,n}$ exhibit instability by blow up. 
 
 We should on the other hand point out that even for the classical case of the quantum harmonic oscillator, \eqref{10}, the (linear and non-linear) stability of the (unique) waves satisfying \eqref{20} is not fully understood. That is, the following question is open, to the best of our knowledge: for solutions of \eqref{20}, with $1+\f{4}{n}<p<p_n^*$, determine the set of $\om$, for which $\phi_\om$ is dynamically stable. Due to the results of Ohta and collaborators,  \cite{Fu, FuOh1, FuOh2, Ohta}, it seems natural to  conjecture the following.
 \begin{conjecture}
 	\label{con:1} Let $n\geq 1$. 
 	Show that for every $p:1+\f{4}{n}<p<p_n^*$, there exists $\om=\om_{p,n}$ so that the unique solution of \eqref{20} is stable whenever $-n<\om\leq \om_{p,n}$ and unstable in the regime $\om>\om_{p,n}$. 
 	\end{conjecture}
 	Such a result would be immensely interesting, especially if it turns out that Conjecture \ref{con:1} does not hold  and hence there is more than one turning point in the stability behavior of the waves. 
 \subsection{Main results}
Regarding the construction of the waves,  we  study the constrained minimization problem 
 \begin{equation}
 \label{100}
 \inf\limits_{\int_{\rn} |u(x)|^2 dx=\la} E[u].
 \end{equation}
 for every $\la>0$. In other words, we will be seeking to minimize the energy for a fixed $L^2$ norm. 
 The constrained minimizers to these problems, if they exists,  are usually referred to as normalized waves. 
 The following is the main existence result of the paper.  
 \begin{theorem}
 	\label{theo:10} 
 	Let $n\geq 1, s\in (0,1]$, $\la>0$, $1<p<1+\f{4s}{n}$ and $V$ is a trapping potential, as defined above.   Then, the constrained minimization problem \eqref{100} has a solution $\phi$, a normalized ground state. Moreover, $\phi\in X^s$ is bell-shaped function, which satisfies the Euler-Lagrange equation \eqref{a:20}, in a distributional sense, with some  $\om=\om_\la$. 
 \end{theorem}
 {\bf Note:} We establish better {\it a posteriori} smoothness and decay results for $\phi$, see Proposition \ref{prop:apost} below. 
 
 Next, we state our results on  the stability of the waves.  Before we move on with the actual statement, we shall need to discuss the related issue of global well-posedness and energy conservation, which is crucial in the orbital stability considerations.  Note that such results are available in the literature, especially in the classical case $s=1$, but definitely not in the generality of potentials that we would like to consider herein. Then, there is the more delicate issue of (formally) conserved  quantities, e.g. $E, P$, in particular the level of regularity needed for the data that is required in order to ensure the actual conservation of energy and $L^2$ norm along the evolution. These subtle points  go beyond the scope of the paper, and more in depth research is required for their full understanding.  For the purposes of this paper,  we assume the said well-posedness (and conservation laws) for the time evolution of \eqref{a:10}. More precisely, 
 \begin{definition}
 	\label{defi:global} We say that the fractional semilinear Schr\"odinger equation \eqref{a:10} is globally well-posed and conserves energy, if  every initial data $u_0\in H^s[\rn]$ produces unique global solution $u(t, \cdot)\in C([0,T], H^s(\rn))$ for each $T>0$ and 
 	\begin{enumerate} 
 		\item the solution map $u_0\to u(t, \cdot)$ is continuous in the norm of $C([0,T], H^s(\rn))$ for small enough times $T$. 
 		\item The energy $E[u]$ and the $P[u]$ are conserved globally in time, that is for each $t>0$, \\ $E[u(t)]=E[u_0]$,  		$P[u(t)]=P[u_0]$. 
 	\end{enumerate}
 \end{definition}
 {\bf Note:} For our purposes, it   suffices to  assume these properties only close to solitons. Note that these assumptions are only needed for the statement of orbital stability of the waves. 
 
 We have the following result regarding the stability of the waves. 
 \begin{theorem}
 	\label{theo:norm}
 	For $n\geq 1, s\in (0,1]$, $\la>0$, $1<p<1+\f{4s}{n}$, the  normalized ground states $\phi$ of the Schr\"odinger equation \eqref{a:10}, with $\|\phi\|^2=\la$, are non-degenerate, in the sense that  $$
 	\cl_+:=(-\De)^s+ V(x) + \om_\la - p \phi_\la^{p-1}
 	$$
 	 has a trivial kernel, i.e. $Ker[\cl_+]=\{0\}$.

 	Finally, assuming global well-posedness and energy conservation, in the sense of Definition \ref{defi:global}, the waves $e^{- i \om t} \phi$ are strongly orbitally stable in the $H^s$ norm.   
 	More  precisely, for all $\eps>0$, there is $\de>0$, so that whenever 
 	$\|u_0-\phi_\om\|_{H^s(\rn)}<\de$, one has 
 	$$
 	\sup_{t>0} \inf_{\theta\in\rone}  \|e^{i \theta} u(t,x) - e^{-i \om t } \phi_\om\|_{H^s(\rn)}<\eps. 
 	$$
 \end{theorem}
  {\bf Remarks:} 
  \begin{itemize}
  	\item The results of Theorem \ref{theo:norm} directly generalize the classical results for the quantum harmonic oscillator model, $s=1$, $V(x)=|x|^2$. 
  	\item The uniqueness of the wave $\phi$, both as a solution of the profile equation to \eqref{a:10}  and as a constrained minimizer of \eqref{100}  is left as an open problem. Clearly, uniqueness in the PDE context is harder than uniqueness of minimizers. 
  	\item We feel comfortable conjecturing a result similar to Conjecture \ref{con:1}. Indeed, at this point the question is wide open, even for values of $\om$ close to the threshold $:-\si_0(\ch)$ as well as large values of $\om$.  
  \end{itemize}
 {\bf Acknowledgement:} We would like to thank our frequent collaborator Sevdzhan Hakkaev for numerous insightful conversations on these topics.  
 
 The paper is organized as follows. In Section \ref{sec:2}, we present some background material, such as rearrangement inequalities, Szeg\"o's inequality (for fractional Laplacians), subspaces of spherical harmonics and relations to spectral theory,  among others. Most of which is well-known, although we present somewhat concise versions/corollaries of the actual results in the literature, which better suit our purposes. In Section \ref{sec:3}, we give the details of the variational construction.  In Section \ref{sec:4}, we first provide a generalization of the Sturm oscillation theorem for the second eigenfunction, recently established in \cite{FLS}, which is then used to establish the non-degeneracy of the wave. We finish Section \ref{sec:4} with a proof of orbital stability of the waves. Finally, in the Appendix, we provide a detailed proof of Proposition \ref{prop:apost}, which yields additional {\it a posteriori} smoothness   properties  of the waves. These  are needed in the arguments, but they may be of independent interest as well.

\section{Preliminaries}
\label{sec:2}
In this section, we collect some preliminary results (as well as some straightforward, mostly well-known calculations), which will be helpful in the sequel. We  introduce some notions, definitions  and notations. 

\subsection{Function spaces and the fractional Laplacian} 

We use the Fourier transform and its inverse in the form 
$$
\hat{f}(\xi)=\int_{\rn} f(x) e^{- i  x\xi} dx, f(x)= (2\pi)^{-n} \int_{\rn} \hat{f}(\xi) e^{ i  x\xi} d\xi
$$
The operator $(-\De)^s$ is defined via its transform as follows $\widehat{(-\De)^s f}(\xi)=|\xi|^{2s} \hat{f}(\xi) $. In particular, we use the notation $|\nabla|=\sqrt{-\De}$. The Sobolev spaces are defined as the closure of the Schwartz functions in $\|f\|_{W^{s,p}}:=\| (Id-\De)^{s/2} f\|_{L^p}$, where $s\in \rone,  1<p<\infty$. 
The Green's function of $((-\De)^s+\la)$ was constructed for example in \cite{FLS}, see Lemma C1 in Appendix C. More precisely, with the notation $\hat{G}_\la(\xi)=\f{1}{|\xi|^{2s}+\la}, \la>0$, there is the representation 
\begin{equation}
\label{eq:g} 
((-\De)^s+\la)^{-1} f(x)=\int_{\rn} G_\la(x-y) f(y) dy,
\end{equation}
where the function $G_\la$ satisfies the following 
\begin{itemize}
	\item $G_\la$ is bell-shaped on $\rn$, $G\in C^\infty(\rn\setminus\{0\})$
	\item $G_\la\in L^r(\rn): 1-\f{1}{r}<\f{2s}{n}$. 
\end{itemize}
\subsection{Rearrangement inequalities} 
Recall the rearrangement inequalities 
\begin{equation}
\label{r:10} 
\int_{\rn}  f(x) g(x) dx\leq \int_{\rn}  f^*(x) g^*(x) dx
\end{equation}
and in addition, for a non-decreasing  function $W$, 
\begin{equation}
\label{r:20} 
\int_{\rn}  W(x) f(x) dx\geq \int_{\rn} W(x) f^*(x) dx
\end{equation}
The following result is sometimes referred to as Fractional Polya-Szeg\"o inequality, for which one can consult the recent work \cite{FL} or the direct and easy proof, which can be found in  Proposition 3, in   \cite{FSS}. 
\begin{proposition}
	\label{prop:11}
	Let $s \in (0,1]$, $n\geq 1$.  Then, for all functions $u\in H^s(\rn)$, we have that its decreasing rearrangement $u^*\in H^s(\rn)$ and moreover
	\begin{equation}
	\label{PS}
	\||\nabla|^s u\|_{L^2(\rn)}\geq \||\nabla|^s u^*\|_{L^2(\rn)}.
	\end{equation}
	In addition, equality is achieved if and only if there exists $x_0\in \rn$ and a decreasing  function $\rho:\rone_+ \to \rone_+$, so that $u(x)=\rho(|x-x_0|)$.
\end{proposition}
Next, we need to discuss  the operator $\ch=(-\De)^s+V$, where  $V$ trapping potential, as assumed above.  To that end, we start with a brief  introduction of  the spaces of spherical harmonics. 
\subsection{Spherical harmonics and representations of  fractional Schr\"odinger operators} It is well-known that the Laplacian on $\rn$  in spherical coordinates is given by 
$$
\De= \p_{rr}+\f{n-1}{r}\p_r+\f{\De_{\sn}}{r^2}.
$$
The spherical Laplacian $\De_{\sn}$ has only point spectrum, in fact $\si(-\De_{\sn})=\{l(l+n-2), l=0,1, \ldots\}$, where each eigenvalue has a subspace of eigenvectors corresponding to $l(l+n-2)$,  $\cx_l\subset L^2(\sn)$, which gives rise to the orthogonal decomposition $ L^2(\sn)=\oplus_{l=0}^\infty  \cx_l$.  Moreover, $\cx_0=span[1]$, while $\cx_1=span[\f{x_j}{r}, j=1, \ldots,n]$. Denote $\cx_{\geq 1}:=\oplus_{l=1}^\infty  \cx_l$, which induces the representation 
$$
L^2(\rn) = L^2(r^{n-1} dr, \cx_0)\oplus L^2(r^{n-1} dr, \cx_{\geq 1})
$$
Thus, we introduce the radial subspace $L^2_{rad}:=L^2(r^{n-1} dr, \cx_0)$. Note that 
$$
-\De|_{L^2_{rad}}=-\p_{rr}-\f{n-1}{r}\p_r,
$$
while 
$$
-\De|_{L^2(r^{n-1} dr, \cx_{\geq 1})}\geq -\p_{rr}-\f{n-1}{r}\p_r+\f{n-1}{r^2}.
$$
For every Banach space $X\hookrightarrow L^2(\rn)$, we denote $X_{rad}:=X\cap L^2_{rad}$. 

For the operators under consideration, $\ch=(-\De)^s+V$, since $V$ is radial, we see that $\ch$ acts invariantly on $L^2(r^{n-1} dr, \cx_{l})$ for each $l$. A moment thought reveals the action of $\ch$ on each such subspace is $\ch_l: L^2(r^{n-1} dr, \cx_{l})\to L^2(r^{n-1} dr, \cx_{l})$, given by the formula 
$$
\ch_l[g Y_l]=\left(\left(-\p_{rr}-\f{n-1}{r}\p_r+\f{l(l+n-2)}{r^2}\right)^{s} g+V g\right)Y_l, 
$$
where $Y_l\in \cx_l$, $g\in L^2_{rad}$. So, 
$$
\ch=\oplus_{l=0}^\infty \ch_l:\oplus_{l=0}^\infty  L^2(r^{n-1} dr, \cx_{l}) \to \oplus_{l=0}^\infty  L^2(r^{n-1} dr, \cx_{l}). 
$$
We shall use the notation, $\ch_{\geq 1}:=\oplus_{l=1}^\infty \ch_l$ for the operator $\ch$ restricted to  $ \oplus_{l=1}^\infty  L^2(r^{n-1} dr, \cx_{l})$. 
 Clearly, the  operator $\ch_l$ is unitarily equivalent to the following operator, denoted again by $\ch_l$, 
$$
\ch_l =\left(-\p_{rr}-\f{n-1}{r}\p_r+\f{l(l+n-2)}{r^2}\right)^{s}+V, 
$$
acting on $L^2_{rad}$, with domain $D(\ch_l)=D(\ch)\cap  L^2(r^{n-1} dr, \cx_{l})$. 
It is clear that  
$$
\si(\ch)=\cup_{l=0}^\infty \si(\ch_l). 
$$
and   $\ch_0<\ch_1<\ch_2<\ldots $. 

Sometimes, e.g. \cite{FL, FLS}, the spectrum (and more specifically the eigenvalues)  of $\ch_0$ is referred to as  radial spectrum/eigenvalues. We adopt this notation. 

\subsection{Some spectral theory for  $\ch$} 
Assume for this section, that $V$ is a real-valued, bounded from below, but otherwise it is unbounded, with at most polynomial growth. We consider the skew-symmetric quadratic form associated to $\ch$, namely 
$$
Q_\ch(u,v)= \dpr{|\nabla|^s u}{\nabla|^s v}+ \int V(x) u(x) \bar{v}(x) dx. 
$$
with form domain\footnote{Due to the polynomial growth assumption for $V$, Schwartz functions are a reliable dense set in all the spaces that we introduce} $H^s(\rn)\cap L^2(V(x) dx)$. Clearly, this can be extended to a self-adjoint operator, with domain $H^{2s}(\rn)\cap L^2(V^2(x) dx)$. 

Clearly, for large enough $M$, say $\inf V(x)>-M$, we have $(-\De)^s+V+2M\geq (-\De)^s+M>0$, so 
$
0<((-\De)^s+V+2M)^{-1}<((-\De)^s+M)^{-1}$ and also $0<((-\De)^s+V+2M)^{-2}<((-\De)^s+M)^{-2}.
$
In particular, 
\begin{equation}
\label{eq:g10} 
\|((-\De)^s+V+2M)^{-1}f\|_{L^2}\leq \|((-\De)^s+M)^{-1}f\|_{L^2}\leq C \|f\|_{H^{-2s}}. 
\end{equation}
From \eqref{eq:g10}, we have that 
$((-\De)^s+V+2M)^{-1}:H^{-2s}(\rn)\to L^2(\rn)$. By duality, we also have $((-\De)^s+V+2M)^{-1}: L^2(\rn)\to H^{2s}(\rn)$ or 
\begin{equation}
\label{eq:g12} 
\|((-\De)^s+V+2M)^{-1}g\|_{H^{2s}}\leq \|g\|_{L^2}. 
\end{equation}
Let us formulate the results in a lemma, which may be useful in other situations. 
\begin{lemma}
	Assume that $n\geq 2$, $s\in (0,1]$ and $V$ is a continuous function, bounded from below. Then, for each $a\in [0,1]$ and for all large enough $N$, we have the bounds 
 \begin{equation}
 \label{eq:g15} 
 \|((-\De)^s+V+N)^{-1}g\|_{H^{2sa}}\leq C \|g\|_{H^{-2s(1-a)}}. 
 \end{equation}
\end{lemma}
{\bf Note:} The estimate \eqref{eq:g15} follows by interpolation between the estimates \eqref{eq:g10} and \eqref{eq:g12}. 
Since in addition $((-\De)^s+V+2M)^{-1}:L^2\to D(\ch)$,  by Kolmogorov-Relich's compactness criteria, $D(\ch)=H^{2s}(\rn)\cap L^2(V^2(x) dx)$ compactly embeds into $L^2(\rn)$, it follows that all $\si(\ch)$ is eigenvalues of finite multiplicity. In addition, these are  sequence of  reals 
$$
\si_0(\ch)\leq \ldots \si_n(\ch)\leq \ldots 
$$
with $\lim_n \si_n(\ch)=\infty$. 
By the Riesz characterization of eigenvalues, we have 
$$
\si_0(\ch)=\inf_{\|u\|=1} [\| |\nabla|^s u\|^2+ \int V(x) u^2(x) dx]. 
$$
By the rearrangement inequalities, more specifically the fractional Polya-Szeg\"o inequality \eqref{PS} and \eqref{r:20}, we conclude the Perron-Frobenius type result, namely that there any   eigenfunction  corresponding to the bottom of the spectrum $\si_0(\ch)$ must be  bell-shaped. 
This implies that $\si_0(\ch)$ is a simple eigenvalue (assuming that there are two different such eigenfucntions, they cannot be orthogonal) and its eigenfunction is positive. 

There is much richer theory concerning the spectrum (and the related eigenfunctions) for $\ch$. Indeed, in the classical case of the Laplacian, i.e. $s=1$ and bounded potentials and one spatial dimension, the Sturm-Liouville theory applies and one has pretty satisfactory theory - every eigenvalue $\si_j(\ch)$  is simple and each eigenfunction has exactly $j$ sign changes. In the recent work, \cite{FL}, the authors have extended this to the case $s\in (0,1)$, still in the one dimensional case. In a subsequent development, \cite{FLS} have extended this to  higher dimensions - such a result is now valid for the radial eigenvalues only and then only for $j=0,1$. They have shown the following theorem, see Theorem 2.3, \cite{FLS}. 
\begin{theorem}(Frank-Lenzmann-Silvestre, Theorem 2.3, \cite{FLS})
	\label{FLS:10} 
	Let $n\geq 1, s\in (0,1]$ and $W$ satisfies 
	\begin{itemize}
		\item $W=W(|x|)$ and $W$ is non-decreasing in $|x|$, 
		\item $W\in L^\infty(\rn)$, $W\in C^\ga, \ga>\max(0, 1-2s)$. That is 
		$$
		|W(x)-W(y)|\leq C |x-y|^\ga.
		$$
	\end{itemize}
	Then, assume that $H=(-\De)^s+W$ has at least two radial eigenvalues $E_0<E_1<\inf \si_{ess}(H)$. 
	
	Then, the corresponding eigenfunction $\Psi_1: \ch \Psi_1=E_1 \Psi_1$ has exactly one change of sign. That is, there exists $r_0\in (0, \infty)$, so that $\Psi_1(r)<0, r\in (0, r_0)$ and $\Psi_1(r)>0, r\in (r_0, \infty)$. 
\end{theorem}

\subsection{The linearized problem for the solitary waves $\phi_\om$} 
\label{sec:2.1}
We now formally state the stability problem for the ground states of \eqref{10}. Namely, we take ansatz in the form 
$$
u(t,x)=e^{- i \om t} (\phi_\om(x)+ v(t,x)),
$$
and plug in the equation \eqref{a:10}. After ignoring all terms in the form $O(v^2)$ and taking a real and imaginary parts ( namely  $v=v_1+i v_2$), we arrive at the following linearized problem 
\begin{equation}
\label{a:60}
\left|
\begin{array}{l}
-\p_t v_2 +((-\De)^s +V(x)+\om) v_1 - p \phi^{p-1} v_1 = 0\\
\p_t v_1 +((-\De)^s +V(x)+\om) v_1 -  \phi^{p-1} v_2 = 0
\end{array}
\right.
\end{equation}
Introducing the linearized self-adjoint operators 
\begin{eqnarray*}
	\cl_+  &=&  (-\De)^s +V+\om  - p \phi^{p-1},\\
	\cl_-  &=&   (-\De)^s +V+\om  -   \phi^{p-1}
\end{eqnarray*}
and the assignments $\vec{v}(t,x)=\left(\begin{array}{c} v_1 \\ v_2 \end{array}\right)\to e^{\la t} \vec{v}(x)$, 
$\cl:= \left(\begin{array}{cc} \cl_+ &  0 \\ 0 & \cl_- \end{array}\right)$, 
$\cj=\left(\begin{array}{cc} 0 & -1  \\ 1 & 0 \end{array}\right)$ allow us to rewrite the eigenvalue problem \eqref{a:60} in the standard form 
\begin{equation}
\label{a:70}
\cj \cl \vec{v}=\la \vec{v}
\end{equation}

\section{Existence of the ground states}
\label{sec:3}
We give  the  variational construction of the ground states. 
\subsection{Variational construction}
\label{sec:3.1}
\begin{proposition}
\label{theo:105}
Let  $s \in (0, 1])$, $ n\geq 2$ and 
$1<p<1+\f{4s}{n}$. Then, the constrained minimization problem \eqref{100} has a solution $\phi$, which belongs to the energy space 
$H^s(\rn)\cap L^2(V(x) dx)\cap L^{p+1}(\rn)$. All solutions $\phi$ are necessarily (a translates of) bell-shaped functions, that is there exists 
$x_0\in \rn$, $a\in \rone$ and $\rho:[0, \infty)\to [0, \infty)$, with $\rho$ decreasing, so that $\phi(x)=a \rho(|x-x_0|)$. 

In addition,    there exists $\om=\om_\la>-\si_0(\ch)$, so that $\phi$ satisfies the Euler-Lagrange equation 
\begin{equation}
\label{110}
(-\De)^s  \phi+V(x)  \phi -  \phi^p+\om_\la \phi=0.
\end{equation} 
\end{proposition}
\begin{proof}
First, we show that the minimization problem \eqref{100} is bounded from below, that is 
$$
\inf\limits_{\int_{\rn} |u(x)|^2 dx=\la} \ce[u]\geq C_\la>-\infty.
$$
Indeed, by Sobolev embedding, we have 
$$
\|u\|_{L^{p+1}(\rn)}^{p+1}\leq C_p \|u\|_{\dot{H}^{n(\f{1}{2}-\f{1}{p+1})}}^{p+1}\leq C_p 
\| |\nabla^s u\|^{\f{n(p-1)}{2 s}} \|u\|^{p+1-\f{n(p-1)}{2s}}=
C_{p,\la} \|\nabla u\|^{\f{n(p-1)}{2 s}}.
$$
Noting that $\f{n(p-1)}{2s}<2$ (since $p<1+\f{4 s}{n}$), we conclude that 
\begin{eqnarray*}
\ce[u] &\geq &  \f{1}{2} [\||\nabla|^s u\|^2+\int V(x) u^2(x) dx] - \f{C_{p,\la}}{p+1} \|\nabla u\|^{\f{n(p-1)}{2}}\geq \\
&\geq &  \f{1}{4} [\||\nabla|^s u\|^2+\int V(x) u^2(x) dx ] -  B_{p,\la}>-\infty. 
\end{eqnarray*}
In particular, for the elements of the constrained set, that is  $\|u\|_{L^2}^2=\la$, there exists a constant $C_\la$, so that 
\begin{equation}
\label{130} 
\||\nabla|^s u\|^2+\int V(x) u^2(x) dx\leq   M_\la. 
\end{equation}
We now apply the theory of decreasing rearrangements for functions on $\rn$. Indeed, by the fractional Polya-Szego inequality, \eqref{PS}, we have $\||\nabla|^s u\|^2\geq \| |\nabla|^s u^*\|_{L^2}^2$. In addition,   by \eqref{r:20}, 
 $$
 \int_{\rn}  V(x)  |u(x)|^2 dx \geq \int_{\rn}  V(x)  |u^*(x)|^2 dx,
 $$
 while $\|u\|_{L^2}=\|u^*\|_{L^2}, \|u\|_{L^{p+1}}=\|u^*\|_{L^{p+1}}$. All in all, it follows that 
$
\ce[u]\geq \ce[u^*],
$
 while the constraint $\int |u^*(x)|^2 dx=\la$ still holds. Moreover, in the Polya-Szeg\"o inequality,  equality is only achieved, if $u(x)=\rho(|x-x_0|)$ for some decreasing function $\rho:\rone_+\to \rone_+$. Thus, we draw the conclusion that the minimization problem \eqref{100} has only bell-shaped solutions (if any!), modulo translations. So, we can concentrate from now on, on the bell-shaped functions only. 
 
 Take a minimizing sequence (of bell-shaped functions) $u_k\in H^s(\rn)\cap L^2(V(x)  dx)$. Denoting 
 \begin{equation}
 \label{200}
 m(\la):= \inf\limits_{\int_{\rn} |u(x)|^2 dx=\la} \ce[u],
 \end{equation}
we have that $\lim_k \ce[u_k]=m(\la)$, with $\int |u_k(x)|^2 dx=\la$. From \eqref{130}, we have that 
$\sup_k \| |\nabla|^s u_k\|<M_\la$. 
We claim that $\{u_k\}$ is a compact sequence in $L^{p+1}$. Indeed, it is bounded in $L^{p+1}$, 
from the Sobolev embedding $H^s\hookrightarrow  L^{p+1}$.  By the Kolmogorov-Riesz compactness  criterium, compactness in $L^{p+1}$ follows from the estimate 
$$
\int V(x) u_k^2(x) dx\leq   M_\la,
$$
since $\lim_{x\to \infty} V(x)=\infty$. 
But  since $u_k$ is bell-shaped and $V$ is non-decreasing, 
$$
M_\la\geq \sup_k  \int_{\rn} V(x) |u_k(x)|^2  dx \geq V(R) \int_{|x|<R} |u_k(x)|^2 dx\geq c_n V(R) R^{n} |u_k(z_0)|^2
$$
for every integer $k$,  every $R>0$ and $z_0: |z_0|=R$. It follows that $|u_k(x)|\leq \f{M_\la}{c_n} V(R)^{-1/2}  |R|^{-n/2}$. Thus, 
$$
  \int_{|x|>R} |u_k(x)|^{p+1} dx\leq \left(\f{M_\la}{c_n \sqrt{V(R)}}\right)^{p+1}   \int_{|x|>R} |x|^{-\f{n}{2}(p+1)} dx \leq c_{\la, n,p}  
  R^{-\f{n(p-1)}{2}}. 
$$
It follows that $\{u_k\}$ is compact in $L^{p+1}(\rn)$. Similarly, $\{u_k\}$ is compact in $L^2(\rn)$, since in addition to being bounded in $H^s(\rn)$
$$
 \int_{|x|>R} |u_k(x)|^2 dx\leq   \f{1}{V(R)} \int_{|x|>R} V(x) |u_k(x)|^2 dx \leq \f{M_\la}{V(R)}.  
$$
Thus, we select a subsequence $u_{k_j}\to \phi$ in $L^{p+1}\cap L^2$, while simultaneously converging weakly in 
$H^s(\rn)\cap L^2(V(x)dx)$. By the lower semi-continuity of norms with respect to weak convergence 
$$
m(\la)=\liminf_j \ce[u_{k_j}]\geq \ce[\phi],
$$
 while $\int \phi^2(x) dx=\lim_j \int u_{k_j}^2(x) dx=\la$. We now see that it must be that 
 $$
 \liminf_j \ce[u_{k_j}]=\lim_j \ce[u_{k_j}]= \ce[\phi], 
 $$
 otherwise one gets a contradiction with the definition of $m(\la)$. Thus, $\phi$ is a solution to \eqref{100} and $m(\la)=\ce[\phi]$. 
 It now remains to derive the Euler-Lagrange equation for $\phi$.  Set for any $\eps\in \rone$ and a test function $h$, 
 \begin{equation}
 \label{302} 
g(\eps) = \ce\left( \sqrt{\la}  \f{\phi+\eps h}{\|\phi+\eps h \|} \right)\geq g(0)=\ce(\phi). 
\end{equation}
We now need to expand $g(\eps)$ in powers of $\eps$, for small $\eps$. To this end, observe that for any  $q$, we have 
 \begin{eqnarray*}
 \|\phi+\eps h \|^q &=&  (\la+2\eps\dpr{\phi}{h}+\eps^2\|h\|^2)^{q/2} =   
 \la^{q/2}\left(1+\eps \f{q\dpr{\phi}{h}}{\la}   +O(\eps^2)  \right)=\\
 &=& \la^{q/2}+\eps q \la^{q/2-1} \dpr{\phi}{h}+O(\eps^2).
\end{eqnarray*}
 Thus\footnote{For the purposes of the derivation of the Euler-Lagrange  equation, the operator $(-\De)^s$ applied on $\phi$ should be understood  in a distributional sense, since {\it a priori}, we only know that $\phi\in H^s(\rn)$. Eventually, we have that $\phi \in H^{2s}(\rn)$, so this will not be an issue.}
 \begin{eqnarray*}
  & &   \f{\la}{2  \|\phi+\eps h \|^2}  \int_{\rn} [| |\nabla|^s  (\phi+\eps h)|^2+ V(x)|\phi+\eps h|^2] dx = \\
  &=& \f{1}{2} \left[ (\| |\nabla|^s \phi\|^2 + \int V(x) \phi^2(x) dx)+2 \eps \dpr{(-\De)^s
  	 \phi+ V(x) \phi}{h}+  
  O(\eps^2) \right] \left[1-2\f{\eps}{\la} \dpr{\phi}{h}+O(\eps^2)\right] = \\
  &=&  \f{1}{2} (\||\nabla|^s \phi\|^2  + \int V(x) \phi^2(x)dx)+\eps\dpr{(-\De)^s \phi+V(x)  \phi-\f{  \||\nabla|^s \phi\|^2 + \int V(x) \phi^2(x)}{\la} \phi}{h}+O(\eps^2).
  \end{eqnarray*}
  In addition, 
  \begin{eqnarray*}
  & & \f{\la^{\f{p+1}{2}}}{(p+1) \|\phi+\eps h \|^{p+1} }\int_{\rn} |\phi(x)+\eps h(x)|^{p+1} dx   = \\
  &=&\f{1}{p+1}[\int\phi^{p+1}(x) dx + \eps(p+1) \dpr{\phi^p}{h}+O(\eps^2)][1-(p+1) \f{\eps}{\la} \dpr{\phi}{h}+O(\eps^2)]= \\
  &=&  \f{1}{p+1} \int\phi^{p+1} dx+\eps[\dpr{\phi^p}{h} -\f{\dpr{\phi}{h}}{\la}  \int\phi^{p+1} dx]+O(\eps^2).
  \end{eqnarray*}
  Putting the last two formulas together 
  \begin{eqnarray*}
    \ce\left( \sqrt{\la}  \f{\phi+\eps h}{\|\phi+\eps h \|} \right)=\ce(\phi)+
    \eps[ \dpr{(-\De)^s \phi+V(x) \phi- \phi^p + \om \phi}{h}]+O(\eps^2),
  \end{eqnarray*}
  where
  $$
  \om=-\f{\||\nabla|^s \phi\|^2  +\int V(x) \phi^2(x) dx  -\int \phi^{p+1}(x) dx}{\la}
  $$
  But $\phi$ is a minimizer, implying that $g'(0)=0$, which amounts to the fact that $\phi$ is a distributional  solution of the following PDE, 
  $$
  (-\De)^s \phi+V(x) \phi- \phi^p  +\om \phi=0
  $$
  Finally, let us show that $\om >-\si_0(\ch)$. To do this, just test the Euler-Lagrange equation with the bell-shaped eigenfunction $\Psi_0: \ch \Psi_0=\si_0(\ch) \Psi_0$. We obtain 
  $$
  \dpr{\Psi_0}{\phi^p}=\dpr{ \Psi_0}{(\ch+\om) \phi}=\dpr{ (\ch+\om)\Psi_0}{ \phi}=(\om+\si_0(\ch)) \dpr{\psi_0}{\phi}.
  $$
  It follows that 
  $$
  \om+\si_0(\ch)=\f{\dpr{\Psi_0}{\phi^p}}{\dpr{\Psi_0}{\phi}}>0. 
  $$
With that, the proof of Proposition \ref{theo:105} is complete.

\end{proof}
Next, we shall need to establish an additional {\it a posteriori} smoothness result for $\phi$. . 
\begin{proposition}
	\label{prop:apost} 
	The normalized waves constructed in Proposition \ref{theo:105} are elements of \\ 
	$H^{2s}\cap L^2(V^2(x) dx)$. In particular, $\phi\in D(\ch)$, so the Euler-Lagrange equation is satisfied in the sense of $L^2$ functions. In addition, 
	$\phi\in C^1(\rn)$. 
\end{proposition}
{\bf Note:} One can establish stronger regularity results, by imposing  stronger  regularity on $V$. 

The somewhat technical proof of Proposition \ref{prop:apost} is presented in  the Appendix. 
We now  establish some additional spectral properties of the operators $\cl_\pm$. 
\subsection{Spectral properties of  $\cl_\pm$} 
\begin{proposition}
	\label{prop:37} 
	The operator $\cl_+$ has  exactly one negative eigenvalue and in fact $\cl_+|_{\{\phi\}^\perp}\geq 0$. In addition, $\phi\perp Ker[\cl_+]$. 
	
	On the other hand, $\cl_-\geq 0$, while $\cl_-[\phi]=0$. Finally,  there exists $\de>0$, so that 	$\cl_-|_{\{\phi\}^\perp}\geq \de$. In particular, $Ker[\cl_-]=span[\phi]$. 
\end{proposition}
{\bf Note:} Due to the fact that $\si(\cl_\pm)$ is all discrete eigenvalues, without finite point of accumulation, it follows that there exists $\de>0$, so that 
\begin{equation}
\label{320} 
\cl_+|_{\{\phi, Ker[\cl_+]\}^\perp}\geq \de>0.
\end{equation}
\begin{proof}
	For the proof of $\cl_+|_{\{\phi\}^\perp}\geq 0$, take a test function $h\perp \phi, \|h\|_{L^2}=1$. Similar to the arguments in the derivation of the Euler-Lagrange equation, we will use the fact that the function $g$, defined in \eqref{302}, satisfies $g''(0)\geq 0$, due to the fact that $\phi$ is a constrained minimum. We have the expansions  
	\begin{eqnarray*}
	\|\phi+\eps h\|_{L^2}^q= (\la+\eps^2)^{q/2}=\la^{q/2} + \f{q\la^{q/2-1}}{2} \eps^2+O(\eps^4), 
	\end{eqnarray*}
	and 
\begin{eqnarray*}
	& &   \f{\la}{2  \|\phi+\eps h \|^2}  \int_{\rn} [| |\nabla|^s  (\phi+\eps h)|^2+ V(x)|\phi+\eps h|^2] dx   =
	 \\
	&=& \f{1}{2} \left[ (\| |\nabla|^s \phi\|^2 + \int V(x) \phi^2(x) dx)+2 \eps \dpr{(-\De)^s
		\phi+ V(x) \phi}{h}  \right] \left[1-\f{\eps^2}{\la} \right] + \\
	&+& \f{\eps^2}{2} \left[ \||\nabla|^s h\|^2+ \int V(x) h^2(x) dx \right]+O(\eps^3) = \f{1}{2}  \left(\||\nabla|^s \phi\|^2  + \int V(x) \phi^2(x) dx\right)+\eps \dpr{\phi^p}{h}+	\\
	&+& \f{\eps^2}{2}\left[\||\nabla|^s h\|^2+ \int V(x) h^2(x) dx-\f{\||\nabla|^s \phi\|^2  + \int V(x) \phi^2(x) dx}{\la}\right]+O(\eps^3),
\end{eqnarray*}
where we have used $(-\De)^s
\phi+ V(x) \phi=\phi^p-\om \phi$ and $\phi\perp h$. 
Similarly, 
 \begin{eqnarray*}
 	& & \f{\la^{\f{p+1}{2}}}{(p+1) \|\phi+\eps h \|^{p+1} }\int_{\rn} |\phi(x)+\eps h(x)|^{p+1} dx   = \\
 	&=&\f{1}{p+1}[\|\phi\|_{L^{p+1}}^{p+1} + \eps(p+1) \dpr{\phi^p}{h}+ \eps^2 \f{(p+1) p}{2} \dpr{\phi^{p-1}h}{h}
 	][1-\f{p+1}{2\la} \eps^2]+O(\eps^3)= \\
 	&=&  \f{1}{p+1} \|\phi\|_{L^{p+1}}^{p+1} +\eps\dpr{\phi^p}{h}  + 
 	\f{\eps^2}{2}\left[p \dpr{\phi^{p-1}h}{h} - \f{\|\phi\|_{L^{p+1}}^{p+1}}{\la}\right]+ O(\eps^3).
 \end{eqnarray*}
 Putting it together, we obtain, 
 $$
 g(\eps)=g(0)+\f{\eps^2}{2}\dpr{\cl_+ h}{h} + O(\eps^3),
 $$
where we have used the representation $$\om=-\f{\||\nabla|^s \phi\|^2  + \int V(x) \phi^2(x) dx-\|\phi\|_{L^{p+1}}^{p+1}}{\la}.$$

Thus, $\dpr{\cl_+ h}{h}=g''(0) \geq 0$, so $\cl_+|_{\{\phi\}^\perp}\geq 0$. It follows that $\cl_+$ has at most one negative eigenvalue. On the other hand,  $\cl_+[\phi]=-(p-1)\phi^p$,  which allows us to compute 
$$
\dpr{\cl_+ \phi}{\phi}=-(p-1) \int \phi^{p+1}(x) dx<0.
$$
From this, $\cl_+$ has indeed a negative eigenvalue and since we have established that it was at most one, it is exactly one, $n(\cl_+)=1$. 

Let us now show that $\phi\perp Ker[\cl_+]$. Note that,  under certain conditions on $V$, we will in fact show the non-degeneracy statement, i.e. $Ker[\cl_+]=\{0\}$, which of course would imply that $\phi\perp Ker[\cl_+]$. On the other hand, this is easy to see without any additional assumptions.

 Indeed, take $\psi\in Ker[\cl_+]$. We have that $\psi-\|\phi\|^{-2} \dpr{\psi}{\phi}\phi \perp \phi$, whence 
$$
0\leq \dpr{\cl_+[\psi-\|\phi\|^{-2} \dpr{\psi}{\phi}\phi]}{\psi-\|\phi\|^{-2} \dpr{\psi}{\phi}\phi}=\|\phi\|^{-4}\dpr{\psi}{\phi}^2 \dpr{\cl_+ \phi}{\phi}.
$$
Since $\dpr{\cl_+ \phi}{\phi}<0$, it follows that $\dpr{\psi}{\phi}=0$, otherwise we reach a contradiction.

Regarding the statement for $\cl_-$, it is clear, by inspection that $L_-[\phi]=0$. Taking arbitrary $h: h\perp \phi$, we have 
\begin{equation}
\label{300}
\dpr{\cl_- h}{h}=\dpr{\cl_+ h}{h}+(p-1) \int \phi^{p-1}(x) h^2(x) dx\geq (p-1) \int \phi^{p-1}(x) h^2(x) dx. 
\end{equation}
From this last inequality, it is clear that there is $\de>0$, $L_-|_{\{\phi\}^\perp}\geq \de$. Indeed, if there is another element in $Ker[\cl_-]$, we can take it $h_0\perp \phi, h_0\neq 0:  \cl_-[h_0]=0$. By \eqref{300}, this would imply that $\int \phi^{p-1}(x) h_0^2(x) dx=0$, which is impossible. So, $Ker[\cl_-]=span[\phi]$. 
\end{proof}

\section{Non-degeneracy and orbital stability of the normalized waves}
\label{sec:4}
 We now aim at establishing the non-degeneracy of the waves $\phi$, that is the Schr\"odinger operator 
 $$
 \cl_+= (-\De)^s+V + \om - p \phi^{p-1},
 $$ 
 has trivial kernel, $Ker[\cl_+]=\{0\}$. The main tool, as in the recent works \cite{FL}, \cite{FLS} is the Sturm oscillation theorem for the  second eigenfunction, Theorem \ref{FLS:10}. There are some technical problems associated with that - in our case the potential $W:=V + \om - p \phi^{p-1}$ is not a bounded function, though it is still non-decreasing and of sufficient smoothness\footnote{here, recall that due to Proposition \ref{prop:apost}, $\phi\in C^1(\rn)$, and so $\phi\in C^1(0, \infty)$ as a function of the radial variable}.  Thus, we need to rely on an approximation argument, and the result that we obtain is somewhat weaker, compared to Theorem \ref{FLS:10}. Nevertheless, it will serve our purposes well.
 \subsection{Sturm oscillation estimate for  the second eigenfunction of a fractional Schr\"odinger operator with increasing unbounded potential}
 \begin{proposition}
 	\label{prop:sturm} 
 	Let $W: \lim_{r\to \infty} W(r)=\infty$ be a radial potential, which is non-decreasing and in the class  	$C^{\ga}_{loc.}((0, \infty))$, $\ga>\max(0, 1-2s)$. That is, for each $N$, there is $C_N$, so that for all $0<r<\rho<N$,  
 	$$
 	|W(\rho)-W(r)|\leq C_N |\rho-r|^\ga.
 	$$
Then, the  smallest eigenvalue of $H_W:=(-\De)^s+W$, $E_0$  is simple, with a bell-shaped eigenfunction. Denote the next  radial eigenvalues of $H_W$   as  $E_0<E_1$. Then, $E_1$ has an eigenfunction with exactly one change of sign. 
 \end{proposition}
 \begin{proof}
 	Define 
 	$$
 	W_N:=\left\{ \begin{array}{cc} 
 	W(r) & 0<r<N \\
 	W(N) & r \geq N. 
 	\end{array}
 	\right.
 	$$
 	Thus, $W_N\in  L^\infty\cap C^{0, \ga}$, so it satisfies the assumptions of Theorem \ref{FLS:10}. Since $\lim_{r\to \infty} W(r)=W(N)$, we have by Weyl's theorem that 
 	$\si_{a.c.}(H_N)=[W(N), \infty)$. Note that since $\lim_N W(N)=\infty$, by the variational characterization of the eigenvalues, there will be plenty of finite multiplicity eigenvalues below $W(N)$. We assume henceforth that $N$  is large enough, so that there are at least two eigenvalues below $W(N)$. 
 	
 	 In addition, by the Perron-Frobenius arguments presented earlier, each $H_N$ has a simple eigenvalue at the bottom of its spectrum $E_{0,N}$, with bell-shaped eigenfunctions, which we denote by $\Psi_{0,N}: \|\Psi_{0,N}\|_{L^2}=1$, that is  $H_N  \Psi_{0,N}=E_{0,N} \Psi_{0,N}$. Note that since $W_N\leq W$, we have that $E_{0,N}$ is an increasing sequence and $E_{0,N}\leq E_0$. Moreover, we have 
 	 $$
 	 \| |\nabla|^s \Psi_{0,N}\|^2+ \int W_N(x) \Psi_{0,N}^2(x) dx = E_{0,N}\leq E_0.
 	 $$
 	It follows that for each $M\geq N$, $\|\Psi_{0,M}\|_{\dot{H}^s}\leq E_0$ and 
 	$W(N) \int_{|r|>N} \Psi_{0,M}^2(x) dx \leq E_0$. This implies that $\{\Psi_{0,N}\}_{N=1}^\infty$ is a compact sequence in $L^2(\rn)$, so it has a limit point $\psi_0:=\lim_k \Psi_{0,N_k}$, which we can in addition  take to be a weak limit in $H^s$ of the same sequence.  Thus, $\|\psi_0\|_{L^2}=1$, 
 	$\| |\nabla|^s \psi_{0}\|\leq \liminf_k \| |\nabla|^s \Psi_{0,N_k}\|$. Finally, for each $R>0$, we have 
 	$$
 	\int_{|x|<R}  W(x) \psi_0^2(x) dx=\lim_k 	\int_{|x|<R}  W_{N_k}(x) \Psi_{0,N_k}^2(x) dx\leq 
 	\limsup_N \int W_N(x) \Psi_{0,N}^2(x) dx.
 	$$
 By Fatou's,  $	\int  W(x) \psi_0^2(x) dx 	\leq 
 \limsup_N \int W_N(x) \Psi_{0,N}^2(x) dx$, whence it follows that 
 	$$
 	\| |\nabla|^s \psi_{0}\|^2+	\int  W(x) \psi_0^2(x) dx \leq  \limsup_N [\| |\nabla|^s \Psi_{0,N}\|^2+\int W_N(x) \Psi_{0,N}^2(x) dx]\leq E_0. 
 	$$
 	It follows that $\psi_0$ is an eigenfunction for $H$, corresponding to the eigenvalue $E_0$, and we have equalities above, which means that $\lim_k \|\Psi_{0,N_k}-\psi_0\|_{H^s}=0$. In fact, by running a simple contradiction  argument similar to the one above, we see that in fact 
 	$\lim_N \|\Psi_{0,N}-\psi_0\|_{H^s}=0$. Clearly, $\psi_0$ is a bell-shaped function as well. 
 	
 	Regarding the eigenvalue $E_1$, we run a similar argument to establish that the eigenfunctions of $H_N$ corresponding to $E_{1,N}$, say $\Psi_{1,N}$,   converge to an eigenfunction corresponding to the   eigenvalue $E_1$. Since Theorem \ref{FLS:10} is applicable to $H_N$, we will be able to conclude that {\it there is an eigenfunction $\psi_1$ of $H_W$, which has exactly one change of sign}. Here are the details. 
 	
 	We start  again with the observation that $E_{1,N}\leq E_1$, since $W_N\leq W$. Further, $\Psi_{1,N}: \|\Psi_{1,N}\|_{L^2}=1$, is so that $\Psi_{1,N}\perp \Psi_{0,N}$ and  
 	$$
 	\| |\nabla|^s \Psi_{1,N}\|^2+ \int W_N(x) \Psi_{1,N}^2(x) dx = E_{1,N}\leq E_1.
 	$$
 	By the same reasoning,  $\Psi_{1,N}$ is a compact sequence in $L^2$, let us denote an accumulation point by $\psi_1: \|\psi_1\|=1$, $\lim_{k\to \infty} \|\psi_1 - \Psi_{1,N_k}\|_{L^2}=0$. Again, we can without loss of generality assume that $\psi_1$ is a weak limit of $\{\Psi_{1, N_k}\}_{k=1}^\infty$  in $H^s$, whence 
 	$\| |\nabla|^s \psi_{1}\|\leq \liminf_k \| |\nabla|^s \Psi_{1,N_k}\|$. Similar to the argument above 
 	$$
 		\| |\nabla|^s \psi_{1}\|^2+	\int  W(x) \psi_1^2(x) dx \leq  \limsup_N [\| |\nabla|^s \Psi_{1,N}\|^2+\int W_N(x) \Psi_{1,N}^2(x) dx]\leq E_1. 
 	$$
 	Note that this implies  $\lim_{N} \|\Psi_{1,N} -\Psi_1\|_{H^s(\rn)}=0$. 
 	Finally, 
 	$$
 	\dpr{\psi_1}{\psi_0}=\lim_N 	\dpr{\Psi_{1,N}}{\Psi_{0,N}}=0.
 	$$
 	Thus,  $\psi_1$ is an eigenfunction for $H_W$, corresponding to the eigenvalue $E_1$. 
 	
 	Now, by Theorem \ref{FLS:10}, $\Psi_{1,N}$ are radial functions, which have exactly one sign change, say $r_N\in (0, \infty)$. Without loss of generality (by replacing $\Psi_{1,N}$ to $-\Psi_{1,N}$ if necessary), assume that \\ $\Psi_{1,N}|_{(0,r_N)}>0$, while $\Psi_{1,N}|_{(r_N, \infty)}<0$. We will show that $\psi_1$ also has exactly one sign change\footnote{ Note that here, the {\it a priori} information is only $\psi_1, \Psi_{1,n}\in H^s(\rn)$, so our functions are not even known to be continuous, unless $s>\f{n}{2}$. On the other hand, the property $\psi$ is positive on an interval $(r_0, \infty)$ is easily tested against a positive  test function. That is $\psi>0$ on an interval $I$, if for every non-negative $C^\infty_0(I)$ function, we have $\dpr{\psi}{\chi}>0$. }. 
 	
 	Indeed, it will suffice to show that  $\{r_N\}_{N=1}^\infty$ has a bounded subsequence, converging to $r_0\in (0, \infty)$.  If that is the case, pick $r_{N_k}\to r_0$ and without loss of generality, assume $r_{N_k}\geq r_0$ (otherwise pick a further subsequence of this property, the case $r_{N_k}\leq r_0$ is symmetric). In such a case, we clearly have that for any $\chi\in C^\infty_0((0, r_0))$, $\chi\geq 0$, we have $\dpr{\psi_1}{\chi}=\lim_k \dpr{\Psi_{1,N_k}}{\chi}\geq 0$. For $\chi\in C^\infty_0((r_0, \infty))$, $\chi\geq 0$, 
 	we have 
 	$$
 	\dpr{\psi_1}{\chi}=\lim_k \dpr{\Psi_{1,N_k}}{\chi}=
 	\lim_k[\int_{|x|\geq r_{N_k}} \Psi_{1,N_k} \chi(x) dx+ \int_{r_0<|x|\leq r_{N_k}} \Psi_{1,N_k} \chi(x) dx]\leq 0, 
 	$$
 since  the second term converges to zero, while the first one is non-positive. 
 
 Thus, it remains to show that $r_N$ has a bounded subsequence, converging to $r_0\in (0, \infty)$. Indeed, otherwise, we have to refute two alternatives - one is that $r_N\to \infty$, while the other is $r_N\to 0$. Assuming $\lim_N r_N=\infty$, we have for any $\chi\in C^\infty_0$, 
 $$
 \dpr{\psi_1}{\chi}= \lim_N \dpr{\Psi_{1,N}}{\chi} \geq 0. 
 $$
 	It follows that $\psi_1\geq 0$, which is a contradiction, since $\dpr{\psi_1}{\psi_0}=0$ (as  eigenfunctions  of $H_W$), while $\psi_1\geq 0$ and $\psi_0$ is bell-shaped. Similarly, if $r_N\to 0$, we conclude 
 	 $$
 	 \dpr{\psi_1}{\chi}= \lim_N \dpr{\Psi_{1,N}}{\chi} \leq 0. 
 	 $$
 	whence $\psi_1\leq 0$, again in contradiction with  $\dpr{\psi_1}{\psi_0}=0$ and $\psi_0$ - bell-shaped. 
 \end{proof}
 \subsection{Non-degeneracy of the wave $\phi$} 
With the results of Proposition \ref{prop:sturm} in hand, we are ready to show the non-degeneracy of $\cl_+$. We know that $\cl_+$ has one simple negative eigenvalue, which is simple, according to Proposition \ref{prop:37}.  

Next, recall that for fractional Schr\"odinger operators like $\cl_+$, there is the decomposition in spherical harmonics 
$$
\cl_+=\cl_{+,0}\oplus \cl_{+,\geq 1}.
$$
 The claim about the non-degeneracy  would thus follow from the two propositions below. 
 
 First, we show that $\cl_{+,0}$, the restriction of $\cl_+$ to the radial subspace,  has exactly one negative eigenvalue and no eigenvalues at zero. 
 \begin{proposition}
 	\label{prop:64} $\si_1(\cl_{+,0})>0$. That is, the second smallest eigenvalue is strictly positive. 	
 \end{proposition}
 For $\cl_+$ restricted to higher harmonics, we show  strict positivity. 
\begin{proposition}
	\label{prop:63} 
	There exists $\de>0$, so that the operator $\cl_{+,\geq 1}\geq \de>0$. That is, the operator $\cl_{+,\geq 1}$ is strictly positive. 
\end{proposition}

\subsubsection{Proof of Proposition \ref{prop:64}} This is just an application of Proposition \ref{prop:sturm}. Indeed, we already know, that there is a negative eigenvalue $E_0$ of $\cl_+$ and hence of $\cl_{+,0}$, which is supported by a bell-shaped eigenfunction. The next eigenvalue $E_1$ cannot be negative, as this will imply that $n(\cl_+)\geq 2$, while we know, that $n(\cl_+)=1$. So, we have to only refute the possibility $E_1=0$. 

Assume for a contradiction $E_1=0$. By Proposition \ref{prop:sturm}, there must be an eigenfunction, $\Psi_1: \cl_{+,0} \Psi_1=0$, so that $\Psi_1$ has exactly one change of sign. Say $\Psi_0(r)<0, r\in (0,r_0)$, while $\Psi_0(r)>0, r\in (r_0, \infty)$. 

On the other hand, we have already checked that $\phi\perp Ker[\cl_+]$. In addition, a direct calculation yields $\cl_{+,0} \phi= -(p-1) \phi^p$, so $\phi^p\perp Ker[\cl_{+,0}]$.  We can construct a linear combination of the two functions, namely 
$$
\Phi:=c_0\phi - \phi^p=\phi(c_0-\phi^{p-1}), c_0:=\phi^{p-1}(r_0), 
$$
which has the property $\Phi(r)<0, r\in (0, r_0)$, $\Phi(r)>0, r\in (r_0, \infty)$, due to the fact that $\phi$ is bell-shaped. On the other hand, $\Phi\perp Ker[\cl_{+,0}]$, so in particular $\dpr{\Phi}{\Psi_1}=0$. But finally, $\Phi \Psi_1\geq 0$ and $\Phi>0$. This provides a contradiction, which finishes the proof of Proposition   \ref{prop:64}.

\subsubsection{Proof of Proposition \ref{prop:63}}
For the Proposition \ref{prop:63}, we start with the   observation that $\cl_{+,\geq 1}\geq 0$, due to the fact that $\cl_+: n(\cl_+)=1$ and the negative eigenvalue has been already accounted for in the radial subspace. Thus, we need to show that zero  is not an eigenvalue for $\cl_{+,\geq 1}$. 

Suppose for a contradiction that zero is an eigenvalue for $\cl_{+,\geq 1}$. We claim that zero then must be an eigenvalue for $\cl_{+,1}$. Assume that this is not the case, then zero is an eigenvalue for $\cl_{+,\geq 2}$, say $\cl_{+,\geq 2} \Phi=0$, where $\Phi=\phi Y_{\geq 2}, Y_{\geq 2}\in \cx_{\geq 2}$. Recalling that $\cl_{+, \geq 2}>\cl_{+, 1}$, it follows that 
$$
\dpr{\cl_{+, 1} \phi}{\phi}<\dpr{\cl_{+, \geq 2} \phi}{\phi}=0, 
$$
whence $\cl_{+,1}$ will have  a negative eigenvalue. In particular, 
 $n(\cl_+)\geq n(\cl_{+,0})+ n(\cl_{+,1})\geq 2$, which is a contradiction. 
 
 Thus, $\cl_{+,1}$ has an eigenvalue at zero, so this  must be clearly the bottom of the spectrum, otherwise again $n(\cl_+)\geq 2$. In addition, its eigenfunctions are be in the form $\Psi_1=\psi_1 Y_1, Y_1\in \cx_1$, so $\Psi_1\in \{\psi_1(r) \f{x_j}{r}, j=1, \ldots, n\}$, so take  $\Psi_1=\psi_1(r) \f{x_1}{r}$.  According to  Lemma C.4, \cite{FLS},  $(-\De_l)^s, s\in (0,1)$ is positivity improving (see also formulas $(C.19)$ and $(C.20)$) and as a consequence 
 $$
 \|(-\De_l)^{s/2} u\|_{L^2}\geq \|(-\De_l)^{s/2} |u| \|_{L^2},
 $$
whence we can conclude that the  radial component $\psi_1$ of $\Psi_1$   is a positive function\footnote{In fact, we can conclude that $\psi_1$ is both positive and decreasing in $(0, \infty)$}, $\psi_1>0$. 

We will show that this leads to a contradiction as well. Namely, take $\p_{x_1}$  in the Euler-Lagrange equation. We obtain the relation 
$$
\cl_{+} (\p_{x_1} \phi)=(-\De)^s \p_{x_1} \phi+ V \p_{x_1} \phi + \om \p_{x_1} \phi- p \phi^{p-1} \p_{x_1} \phi= -\f{\p V}{\p x_1}  \phi=-V'(r) \f{x_1}{r}\phi. 
$$
Taking dot product with $\Psi_1$ yields 
$$
0=\dpr{\p_{x_1} \phi}{\cl_{+} \Psi_1}=\dpr{\cl_{+} (\p_{x_1} \phi)}{\Psi_1}=-\int_0^\infty V'(r) \phi(r) \psi_1(r)  x_1^2 r^{n-3} dr<0, 
$$
since $V'>0$ and all the other integrands are non-negative. This is a contradiction, so Proposition \ref{prop:63} is established as well.

\subsection{Orbital stability} 
Before we set up the problem, let us mention that for this part of it, we assume global well-posedness and conservation of energy per Definition \ref{defi:global}.

 We would also like to change variables in a way that reduces matters a bit. Namely, using the ansatz $u\to e^{- i \om t} u$, we reduce the equation \eqref{a:10} to 
\begin{equation}
\label{a:12} 
i u_t +(-\De)^s u+V(x) u +\om u-|u|^{p-1} u=0, (t,x)\in \rone_+\times \rn,
\end{equation}
which in its current form has the time independent solution $u(t,x)=\phi(x)$. So, orbital stability for the solution $e^{- i \om t} \phi$ for \eqref{a:10} is equivalent to orbital stability for the static solution $\phi$ for \eqref{a:12}. 

That is, we are trying to show that for every $\eps>0$, there exists $\de=\de_\eps$, so that whenever $\|u_0-\phi\|_{H^s(\rn)}<\de$, then the solution of \eqref{a:12} with initial data $u_0$ satisfies 
$$
\sup_{0<t<\infty} \inf_{\theta\in \rone} \|u(t, \cdot)- e^{i \theta} \phi\|_{H^s(\rn)}<\eps.
$$
We argue by contradiction. 
Specifically, assume that there is $\eps_0>0$ and a sequence of initial data, $u_n: \lim_n \|u_n-\phi\|_{H^s(\rn)}=0$, while 
\begin{equation}
\label{410} 
\sup_{0<t<\infty} \inf_{\theta\in \rone} \|u_n(t, \cdot)- e^{i \theta} \phi\|_{H^s(\rn)}\geq \eps_0, 
\end{equation} 
Note the conservation of {\it total energy} for solutions of \eqref{a:12}, namely 
$$
\ce[u]:=\f{1}{2}\left( \int_{\rn}| |\nabla|^s u(t,x)|^2 dx +   \int_{\rn} (V(x)+\om) |u(t,x)|^2 dx\right) - \f{1}{p+1} \int_{\rn}  |u(t,x)|^{p+1} dx,
$$
and in addition $P[u]=\int_{\rn}  |u(t,x)|^2 dx$ is conserved as well. This are our assumptions in Definition \ref{defi:global}! 

Clearly, the Euler-Lagrange equation, satisfied by $\phi$ is equivalent to 
$
\ce'[\phi]=0, 
$
where $\ce'$ is the Gateaux derivative of the functional $\ce$.  
Introduce 
$$
\eps_n:=|\ce[u_n(t)]-\ce[\phi]|+|P[u_n(t)]-P[\phi]|.
$$
Note that by the conservation laws, $\eps_n$ is conserved and hence $\lim_n \eps_n=0$, since 
$\eps_n\leq C \|u_n-\phi\|_{H^s}$. 
Next, for all $\eps>0$, define $t_n=\sup \{\tau>0: \sup_{0<t<\tau} \|u_n(t)-  \phi\|_{H^s}<\eps\}$. We have that all $t_n>0$, by the continuity of the solution maps $u_0\to u(t, \cdot)$ as mappings from $H^s$ into itself. Introduce $u_n(t, \cdot)=v_n(t, \cdot) + i w_n(t, \cdot)$. 

We are now ready to introduce the modulation parameter $\theta_n(t)$ as long as  $\|u_n(t)-  \phi\|_{H^s}<<1$. 
Indeed, taking initially $t\in (0, t_n)$ guarantees that $\|w_n(t)\|_{H^s} \leq \|u_n(t)-\phi\|_{H^s}<\eps$. As  a consequence, $\theta_n(t)$ is defined so that $w_n(t, \cdot)- \sin(\theta_n(t)) \phi \perp \phi$ or equivalently 
\begin{equation}
\label{430} 
\sin(\theta_n(t)) \|\phi\|^2=\dpr{w_n(t)}{\phi}.
\end{equation}
This last equation explicitly defines an unique small solution $\theta_n(t)$ of \eqref{430},  since 
$|\dpr{w_n(t)}{\phi}|\leq \eps \|\phi\|_{L^2}$. With this assignment, and as long as it holds that 
$ \|u_n(t)-\phi\|_{H^s}<\eps$, we have the estimate 
\begin{equation}
\label{450}
\|u_n(t, \cdot)-e^{i \theta_n(t)} \phi\|_{H^s}\leq \|u_n(t, \cdot) -  \phi\|_{H^s}+
|e^{i \theta_n(t)}-1|\|\phi\|_{H^s}\leq C_0\eps,
\end{equation}
where $C_0=C_0(\|\phi\|)$. Define 
$$
T_n=\sup\{\tau: \sup_{0<t<\tau} \|u_n(t, \cdot)-e^{i \theta_n(t)} \phi(\cdot)\|_{H^s}<2C_0 \eps\}. 
$$
Due to \eqref{450}, we have that $T_n>t_n>0$. Note that the construction above holds for all small enough values of $\eps>0$. We will show that for all small enough values of $\eps$ and for all large enough $n$, $T_n=\infty$. This would be in contradiction with \eqref{410}, provided one chooses $\eps<<\eps_0$ and large enough $n$ and the orbital stability will be established accordingly. 

Write for $t\in (0, T_n)$  
$$
\psi_n(t, \cdot):= u_n(t, \cdot) - e^{i \theta_n(t)} \phi=v_n(t, \cdot) - \cos(\theta_n(t)) \phi+ 
i(w_n(t, \cdot) - \sin(\theta_n(t)) \phi).
$$
Note that while $0<t<T_n$, $\|\psi_n(t)\|_{H^s}<2\eps$, according to the definition of $T_n$. Decompose the real and the imaginary part of $w_n$ as follows 
\begin{equation}
\label{460} 
\left(\begin{array}{c} 
v_n(t, \cdot) - \cos(\theta_n(t)) \phi \\
w_n(t, \cdot) - \sin(\theta_n(t)) \phi
\end{array}\right)=\mu_n(t) \left(\begin{array}{c} 
\phi \\ 0
\end{array}\right)+ \left(\begin{array}{c} 
\eta_n(t, \cdot) \\ \zeta_n(t, \cdot)
\end{array}\right), \ \ \left(\begin{array}{c} 
\eta_n(t, \cdot) \\ \zeta_n(t, \cdot)
\end{array}\right)\perp \left(\begin{array}{c} 
\phi \\ 0
\end{array}\right). 
\end{equation}
Note that the condition $\left(\begin{array}{c} 
\eta_n(t, \cdot) \\ \zeta_n(t, \cdot)
\end{array}\right)\perp \left(\begin{array}{c} 
\phi \\ 0
\end{array}\right)$ simply means $\eta_n(t, \cdot)\perp \phi$, while the defining  equation \eqref{430} came from $w_n(t, \cdot)- \sin(\theta_n(t)) \phi \perp \phi$ or equivalently $\zeta_n(t, \cdot)\perp \phi$. 
On the other hand,  
\begin{eqnarray*}
P[u_n(t)] &=& \int_{\rn} |e^{i \theta_n(t)} \phi+\psi_n(t)|^2 dx=P[\phi]+\|\psi_n(t, \cdot)\|_{L^2}^2+ 2\int_{\rn} \phi(x) \Re  [e^{i \theta_n(t)} \psi_n(t, x)] dx
\end{eqnarray*}
But 
\begin{eqnarray*}
  \int \phi(x) \Re  [e^{i \theta_n(t)} \psi_n(t, x)] dx &=& 	\int \phi(x) [\cos(\theta_n)(v_n - \cos(\theta_n) \phi)  -\sin(\theta_n) (w_n - \sin(\theta_n) \phi)] dx = \\
&=&  \mu_n(t) \cos(\theta_n(t)) \|\phi\|^2,
\end{eqnarray*}
due to $\eta_n\perp \phi$ and $w_n - \sin(\theta_n) \phi\perp \phi$. 
It follows that, 
$$
P[u_n(t)] = P[\phi]+\|\psi_n(t, \cdot)\|_{L^2}^2+ 2  \mu_n(t) \cos(\theta_n(t)) \|\phi\|^2,
$$
whence by recalling that $\|\psi_n(t, \cdot)\|_{L^2}\leq 2\eps$, in $t: 0<t<T_n$ 
\begin{equation}
\label{470} 
|\mu_n(t)|\leq \f{|P[u_n(t)]-P[\phi]|+ \|\psi_n(t, \cdot)\|_{L^2}^2}{2\cos(\theta_n(t) \|\phi\|^2}\leq C(\eps_n+\|\psi_n(t, \cdot)\|_{L^2}^2)\leq C(\eps_n+\eps^2), 
\end{equation}
  since $|\theta_n(t)|\leq C_0\eps<<1$ and hence $\cos(\theta_n(t))=1+O(\eps^2)$. 
Next, 
\begin{eqnarray*}
& & \ce[u_n(t)]-\ce[\phi] = \ce[(\cos(\theta_n(t))\phi+\mu_n(t)\phi+\eta_n)+i (\sin(\theta_n(t)\phi+\zeta_n)]-\ce[\phi]= \\
&=&  \f{1}{2}[\dpr{\cl_+ \eta_n}{\eta_n}+\dpr{\cl_- \zeta_n}{\zeta_n}]+O(\eps_n+ \|\eta_n\|_{H^s}^3+\|\zeta_n\|_{H^s}^3+ \eps^3), 
\end{eqnarray*}
where we took into account $\ce'[\phi]=0$, as well as \eqref{470}. 

We now need the important observation that  according tom Proposition \ref{prop:37} for $\cl_-$ , and then  the non-degeneracy for $\cl_+$, 
we have that there exists $\ka>0$, so that for every $\eta\perp \phi, \zeta \perp \phi$
\begin{equation}
\label{400} 
\dpr{\cl_+ \eta}{\eta}\geq \ka \|\eta\|_{H^s}^2, \dpr{\cl_- \zeta}{\zeta}\geq \ka \|\zeta\|_{H^s}^2, 
\end{equation}
The non-coercivity property \eqref{400} allows us to estimate 
$$
\ce[u_n(t)]-\ce[\phi] \geq \ka(\|\eta_n\|_{H^s}^2+\|\zeta_n\|_{H^s}^2) - C (\eps_n+ \|\eta_n\|_{H^s}^3+\|\zeta_n\|_{H^s}^3+ \eps^3). 
$$
Taking into account that $\eps_n\geq |\ce[u_n(t)]-\ce[\phi] |$, we finally arrive at 
\begin{equation}
\label{480} 
\|\eta_n(t, \cdot)\|_{H^s}^2+\|\zeta_n(t, \cdot)\|_{H^s}^2\leq C(\eps_n+ \eps^3+ 
\|\eta_n(t, \cdot)\|_{H^s}^3+\|\zeta_n(t, \cdot)\|_{H^s}^3),
\end{equation}
for every $t\in (0,T_n)$.  Since  for each $t\in (0,T_n)$, 
$\|\eta_n(t, \cdot)\|_{H^s}+\|\zeta_n(t, \cdot)\|_{H^s}<2 C_0 \eps$, we have that from \eqref{480} and for small enough $\eps$, 
$$
C(\|\eta_n(t, \cdot)\|_{H^s}^3+\|\zeta_n(t, \cdot)\|_{H^s}^3)\leq \f{1}{2}(\|\eta_n(t, \cdot)\|_{H^s}^2+\|\zeta_n(t, \cdot)\|_{H^s}^2),
$$
 whence again by \eqref{480}, we can bootstrap it to 
\begin{equation}
\label{490} 
 \|\eta_n(t, \cdot)\|_{H^s}^2+\|\zeta_n(t, \cdot)\|_{H^s}^2\leq C(\eps_n+ \eps^3), \ \ t\in (0,T_n). 
\end{equation}
 This last estimate shows that for small enough $\eps$ and then  large enough $n$ (recall $\lim_n \eps_n=0$), it must be that $T_n=\infty$, by its definition, since $\eps^{3/2}+\sqrt{\eps_n}<<\eps$. 
 This concludes the proof of the orbital stability.

\appendix

\section{A posteriori smoothness and decay: Proof of Proposition \ref{prop:apost}}  
\label{sec:app1} 
 
	We start with the {\it a priori} information from Proposition \ref{theo:105}, that is $\phi$ is bell-shaped and in the class $\phi\in H^s\cap L^2(V(x) dx)$, together with the fact that $\phi$ is a weak solution of \eqref{110}. 
	
	In order to obtain bootstrap this information, we need a representation of $\phi$ from the Euler-Lagrange PDE. Unfortunately, $\phi$ is still only a weak solution of \eqref{110}, as we have pointed out. Instead, define for large enough $N$, 
	$$
	\tilde{\phi}:= ( (-\De)^s+V +\om_\la+N)^{-1}[\phi^p+N\phi].
	$$
	Heuristically, this is the solution of the \eqref{110}, if $\phi$ were a solution in a stronger sense. 
	In fact, it is not even immediately clear in what sense is $\tilde{\phi}$   even defined. Clearly, while 
	$$
	\|( (-\De)^s+V +\om_\la+N)^{-1}[\phi]\|_{L^2}\leq C\|\phi\|_{L^2}
	$$
	 is under control, it is not as easy to control  $( (-\De)^s+V +\om_\la+N)^{-1}[\phi^p]$, since the {\it a priori} information on $\phi^p$ is very weak. Instead, for $n\leq 4s$, we can bound by \eqref{eq:g10} and Sobolev embedding 
	$$
	\| ( (-\De)^s+V +\om_\la+N)^{-1}[\phi^p]\|_{L^2(\rn)}\leq C \|\phi^p\|_{H^{-2s}}\leq C \|\phi^p\|_{L^{\f{p+1}{p}}} = 
	C \|\phi\|_{L^{p+1}}^p.
	$$
	while for $n>4s$, 
	we bound by  \eqref{eq:g10} and by repeated application of Sobolev  embedding 
	$$
	\| ( (-\De)^s+V +\om_\la+N)^{-1}[\phi^p]\|_{L^2(\rn)}\leq C \|\phi^p\|_{H^{-2s}}\leq C  \|\phi^p\|_{L^{2}} \leq C \|\phi\|_{H^s(\rn)}^p.
	$$
	So, $\tilde{\phi}$ is well-defined as an $L^2(\rn)$ function. 
	Consider a test function $h\in H^{2s}\cap L^2(V^2(x) dx)$, 
	\begin{eqnarray*}
		\dpr{\tilde{\phi}}{( (-\De)^s+V +\om_\la+N) h} = \dpr{\phi^p+N\phi}{h}=\dpr{\phi}{(-\De)^s+V +\om_\la+N) h}.
	\end{eqnarray*}	
	It follows that $\dpr{\phi-\tilde{\phi}}{(-\De)^s+V +\om_\la+N) h}=0$. Since the set 
	$\{ (-\De)^s+V +\om_\la+N) h: h\in  H^{2s}\cap L^2(V^2(x) dx)\}$ is dense in $L^2$, we have that $\phi=\tilde{\phi}$ or 
	\begin{equation}
	\label{310} 
	\phi= ( (-\De)^s+V +\om_\la+N)^{-1}[\phi^p+N\phi].
	\end{equation}
We now run a bootstrapping procedure, which will ultimately establish that $\phi\in H^{2s}(\rd)$.  Starting with $\al_0=s$, we define $\al_{k+1}$, as long as $\al_k<2s$. We have for $\al: \al_k<\al\leq 2s$,  by Sobolev embedding, \eqref{eq:g15} and Kato-Ponce estimates 
\begin{eqnarray*}
\|\phi\|_{H^{\al}}&\leq & C[\|\phi\|_{L^2}+\|\phi^p\|_{H^{\al-2s}}]\leq  
C[\|\phi\|_{L^2}+\|\phi\|_{H^{\al_k}} \|\phi^{p-1}\|_{L^{\f{n}{2s+\al_k-\al}}}]=\\
&=& 
C[\|\phi\|_{L^2}+\|\phi\|_{H^{\al_k}} \|\phi\|_{L^{\f{n(p-1)}{2s+\al_k-\al}}}^{p-1}].
\end{eqnarray*}
In the last term, if we make sure that $\f{n(p-1)}{2s+\al_k-\al}\leq p+1$, we will have control of the right-hand side. Given the restriction $p<1+\f{4s}{n}$, this would be satisfied, if 
$$
\al-\al_k\leq \f{4s^2}{n+2s}.
$$
So, we define $\al_{k+1}:=\min(2s, \al_k+\f{4s^2}{n+2s})$, whence we conclude that $\phi\in H^{\al_{k}}$ for each $k$. Clearly, in finitely many iterations, we will reach $\phi\in H^{2s}(\rd)$.  

Furthermore, 	$\phi^p\in L^2$, since 
	$$
	\|\phi^p\|_{L^2}=\|\phi\|_{L^{2p}}^p\leq C \|\phi\|_{H^{2s}}^p,
	$$
	since $p<1+\f{4s}{n}$. It follows from \eqref{eq:g15} that $\phi\in H^{2s}\cap L^2(V^2(x) dx)$ since 
	$$
	\|\phi\|_{H^{2s}\cap L^2(V^2(x) dx)}\leq C [\|\phi\|_{L^2}+ 	\|\phi^p\|_{L^2}]<\infty.
	$$
	Once we have that $V \phi \in L^2$, it is easy to bootstrap even further. Indeed, we will have that the expression $( (-\De)^s+\om+N)^{-1}[(V+N) \phi]$ makes sense as $L^2$ function, which is positive everywhere, for $N$ large enough, as convolution of $G_{\om+N}>0$ and $(V+N) \phi>0$. Hence, we  have 
	$$
	0<\phi=( (-\De)^s+\om+N)^{-1}[\phi^p+2 N \phi- (V+N)\phi]\leq ( (-\De)^s+\om+N)^{-1}[\phi^p+2 N \phi]
	$$
	This last inequality can be now iterated to $\phi\in L^\infty(\rn)$, see p. 1723, \cite{FLS}.  
	
	We now aim at extending this further to Lipschitz continuity. To this end, introduce a smooth and even cut-off function $\chi: supp\chi\subset (-2,2)$, so that $\chi(x)=1, |x|<1$. Let $N>>1$ and $\chi_N(x):=\chi(x/N)$. Multiplying the equation \eqref{a:20} by the cutoff $\chi_N$ and $\phi_N:=\phi(x) \chi_N$, we can rewrite it in the form 
	\begin{equation}
	\label{pol}
		((-\De)^s +\om +M)\phi_N= - V \phi_N+\phi^{p}\chi_N+M \phi_N+[(-\De)^s, \chi_N]\phi.
	\end{equation}
	for any $M$. The operator on the left-hand side is invertible for large enough $M$, and we can write 
		\begin{equation}
		\label{pol1}
	\phi_N=((-\De)^s +\om +M)^{-1}[- V \phi_N+\phi^{p}\chi_N+M \phi_N+[(-\De)^s,\chi_N]\phi].
	\end{equation}
	According to the Mikhlin multplier's theorem, $((-\De)^s +\om +M)^{-1}$ smooths out by $2s$ derivatives in any Sobolev space $W^{\al, p}, 1<p<\infty$. It follows that for any $\al<2s$, 
	$$
	\|\phi_N\|_{W^{\al, p}}\leq C_{\al, p}[ \|V \phi_N\|_{L^p} +\|\phi^{p}\chi_N\|_{L^p} 
	+M \|\phi_N\|_{L^p}+\|[(-\De)^s,\chi_N]\phi\|_{L^p}\leq C_{\al, p},
	$$
	due to the {\it a priori} bounds on $\|\phi\|_{L^p}$, 
	 and the fact that $V$ is bounded on the support of $\chi_N$. 
	Note that we also have used a corollary of the commutator estimates to derive 
 $\|[(-\De)^s,\chi_N] \phi\|_{L^p}\leq C_{N,p, \tilde{p}} \|\phi\|_{L^{\tilde{p}}}, \tilde{p}>p$. It follows that $\phi_N\in W^{2s, p}, p<\infty$ for each $N$. If $2s>1$, there is nothing to do, as   $\phi_N\in W^{1+, p}, p<\infty$, which by Sobolev embedding will imply that $\phi\in C^1$ as required. 
 
 Otherwise,  apply $(-\De)^s$ to \eqref{pol} and then use the inversion formulas as
in \eqref{pol1}. Since $\phi_N\in W^{2s, p}$, we see that (recall that $V\in C^1(\rn)$) 
$$
(-\De)^s [- V \phi_N+\phi^{p}\chi_N+M \phi_N+[(-\De)^s,\chi_N]\phi] \in L^p,
$$
whence $\phi_N\in W^{4s, p}$ and so on. This can be bootstrapped, in finitely many steps to the desired outcome $\phi_N\in W^{1+, p}, p<\infty$, so $\phi\in C^1$. We omit further details.

\end{document}